\documentclass[12pt,a4paper]{amsart}
\usepackage[top=35mm, bottom=35mm, left=30mm, right=30mm]{geometry}
\usepackage[colorlinks=true,citecolor=blue]{hyperref}
\usepackage{amssymb}
\usepackage{amsthm}

\usepackage{enumerate}
\usepackage{verbatim}

\usepackage[looser]{newtxtext}
\usepackage{newtxmath}

\newtheorem{thmx}{Theorem}

\newtheorem{thm}{Theorem}[section]
\newtheorem{cor}[thm]{Corollary}

\newtheorem{lem}[thm]{Lemma}
\newtheorem{prop}[thm]{Proposition}

\theoremstyle{definition}
\newtheorem{defn}[thm]{Definition}
\newtheorem{exam}[thm]{Example}
\newtheorem{rem}[thm]{Remark}
\numberwithin{equation}{section}

\newcommand{\eps}{\varepsilon}
\newcommand{\bbn}{\mathbb{N}}
\newcommand{\bbz}{\mathbb{Z}}

\newcommand{\orb}{\mathrm{Orb}}

\newcommand{\supp}{\mathrm{supp}}

\newcommand{\calf}{\mathcal{F}}
\newcommand{\N}{\mathbb{N}}

\newcommand{\Z}{\mathbb{Z}}
\newcommand{\ra}{\rightarrow}

\newcommand{\ep}{\epsilon}
\newcommand{\calp}{\mathcal{P}}

\DeclareMathOperator{\diam}{diam}

\DeclareMathOperator{\tran}{Tran}

\begin{document}
\title{Broken family sensitivity in transitive systems}
\author{Jian Li and Yini Yang}
\address{Department of Mathematics, Shantou University,
	Shantou, Guangdong, 515063, P.R. China}
\email{lijian09@mail.ustc.edu.cn}
\email{ynyangchs@foxmail.com}
\date{\today}
\thanks{The authors were supported by NNSF of China (12171298, 11871188) and NSF of Guangdong Province (2018B030306024).\\
\indent Y. Yang is the corresponding author.
}

\begin{abstract}
Let $(X,T)$ be a topological dynamical system, $n\geq 2$ and $\calf$ be a Furstenberg family of subsets of $\bbz_+$.
$(X,T)$ is called broken $\calf$-$n$-sensitive
if there exist $\delta>0$ and $F\in\calf$ such that for every opene (non-empty open) subset $U$ of $X$ and every $l\in\bbn$,
there exist $x_1^l,x_2^l,\dotsc,x_n^l\in U$ and $m_l\in \bbz_+$ satisfying
$d(T^k x_i^l, T^k x_j^l)> \delta,\ \forall 1\leq i<j\leq n, k\in m_l+ F\cap[1,l]$.
We investigate broken $\calf$-$n$-sensitivity for the family of all piecewise syndetic subsets ($\calf_{ps}$), the family of all positive upper Banach density subsets ($\calf_{pubd}$) and the family of all infinite subsets ($\calf_{inf}$).
We show that a transitive system $(X,T)$ is broken  $\calf$-$n$-sensitive for $\calf=\calf_{ps}\ \text{or}\ \calf_{pubd}$ if and only if there exists an essential $n$-sensitive tuple which is
an $\calf$-recurrent point of $(X^n, T^{(n)})$; is broken $\calf_{inf}$-$n$-sensitive
if and only if there exists an essential $n$-sensitive tuple $(x_1,x_2,\dotsc,x_n)$
such that
$\limsup_{k\to\infty}\min_{1\leq i<j\leq n}d(T^kx_i,T^kx_j)>0$.
We also obtain specific properties for them by analyzing  the factor maps to their maximal equicontinuous factors. Furthermore, we show examples to distinguish different kinds of broken family sensitivity.
\end{abstract}
\maketitle

\section{Introduction}
Throughout this paper, by a \textit{topological dynamical system}, we mean a pair $(X,T)$, where $X$
is a compact metric space with a  metric $d$ and $T\colon X\to X$ is a continuous surjective map. In the end of the 1970s, the well known notion of sensitivity was proposed by Ruelle \cite{Ruelle1977}, which shows that
small changes in initial conditions may leads to
great difference in the behavior of a system of its long time iterates. To be precise, a topological dynamical system
is called \emph{sensitive}
if there exists $\delta>0$ such that for every
opene (non-empty open) subset $U$, there exist $x,y\in U$ and $m\in\bbn$ such that $d(T^mx, T^my)>\delta$.

\medskip
A topological dynamical system $(X,T)$ is called \emph{equicontinuous} if for every $\varepsilon>0$
there is $\delta>0$ such that
whenever $x,y\in X$ with $d(x,y)<\delta$, $d(T^mx,T^my)<\varepsilon$ for any $m\in \bbn$. Roughly speaking, in an equicontinuous system if two points are close enough then they will  always be close under iterations. So equicontinuous systems have simple dynamical behaviors.
The  interesting
dichotomy theorem proved by Auslander and Yorke \cite{Auslander1980} is as follows: every minimal
system is either equicontinuous or sensitive.

\medskip
In order to investigate some dynamical properties and give explicit characterizations, the researchers
introduced some new kinds of sensitivity and obtain interesting results. In \cite{X05} Xiong introduced a stronger form of sensitivity named $n$-sensitivity.
Shao, Ye and Zhang \cite{SYZ,YZ08} studied $n$-sensitivity extensively, particularly for minimal systems. 
A non-diagonal $n$-tuple $(x_1, x_2,\cdots, x_n) \in X^n$ is called an \emph{$n$-sensitive tuple}
if for any $\eps>0$ and any opene subset $U$ of $X$,
there exist $y_1, y_2,\dotsc, y_n\in U$ and $m\in\bbn$ such that
$d(T^my_i,x_i)<\eps$ for every $i=1,2,\dotsc,n$.
If in addition $x_i\not=x_j$, for any $i\not=j$, 
$(x_1, x_2,\cdots, x_n)$ is called  an \emph{essential $n$-sensitive tuple}.
They show that a transitive system is $n$-sensitive if and only if there exists an essential $n$-sensitive tuple.
Huang, Lu and Ye \cite{HLY11} finally obtained the structure of $n$-sensitivity for minimal systems, more precisely, by the fibers of factor map to the maximal equicontinuous factor.
The notion of sensitivity was generalized by measuring the set of nonnegative
integers for which the sensitivity occurs \cite{HKKPZ16, HKZ18,  LTY15, Li-Yang1, Li-Yang2,LShi14, M07, TZ11, Y18}. We refer the reader to the survey \cite{Li-Ye} for related results.
Let $n\geq 2$.
A dynamical system $(X,T)$ is called \emph{blockily thickly $n$-sensitive}
if there is $\delta>0$
such that for every $k\in \bbn$ and every opene subset $U$ of $X$,
we can find $x_1^k,x_2^k,\dotsc,x_n^k\in U$ and $m_k\in\bbn$
such that
\[
\{m_k,m_k+1,\dotsc,m_k+k\}\subset
\Bigr\{ m\in\bbn\colon
\min_{1\le i<j\le n}d(T^m x_i^k,T^m x_j^k)>\delta\Bigr\}.
\]
In \cite{Y18} Ye and Yu showed that a minimal system is either blockily thickly sensitive or a proximal extension of its maximal equicontinuous factor.
In \cite{Zou17} Zou showed that a minimal system is
blockily thickly $n$-sensitive if and only if for the factor map to its maximal equicontinuous factor, there exist $n$ pairwise distinct points such that any two of them form a distal pair.

\medskip
We find that the previous results of sensitivity are based on the factor map to its maximal equicontinuous factor, which is closely related to the regional proximal relations in minimal systems.
So, a natural question arise:
does there exist some kinds of sensitivity which can be used in the
explicit characterizations of dynamical properties in transitive systems?

In this paper, we propose a new kind of sensitivity to give explicit characterizations of dynamical properties in transitive systems. Inspired by blockily thickly sensitivity we define the new kind of sensitivity as follows:

Let $(X,T)$ be a dynamical system, $n\geq 2$ and $\calf$ be a Furstenberg family of subsets of $\bbz_+$. 
We say that $(X,T)$ is broken  $\calf$-$n$-sensitive
if there exist $\delta>0$ and $F\in\calf$ with the property that for every opene subset $U$ of $X$ and every $l\in\bbn$,
there exist $x_1^{l},x_2^{l},\dotsc,x_n^{l}\in U$ and $m_{l}\in \bbz_+$ such that
\[
\min_{1\le i<j\le n}d(T^k x_i^{l}, T^k x_j^{l})> \delta,\ \forall k\in m_{l}+ F\cap[1,l].
\]
Here we remark that this kind of sensitivity is closely related to the
broken family. Recall that the broken family of $\calf$, denoted by $b\calf$ is as follows: $F\in b\calf$ if and only if there is a sequence $\{a_n\}_{n=1}^{\infty}$ in $\bbz_+$ and $F^{\prime}\in\calf$ such that
$\bigcup_{n=1}^{\infty}(a_n+F^{\prime}\cap[1,n])\subset F$. The idea of broken family
was first proposed in~\cite{Blokh02}
and the authors call the definition block family in \cite{HLY12}. We use ``broken family sensitivity" for this notion. In fact, following from the definitions
blockily thickly $n$-sensitivity is equivalent to broken  $\{\bbn\}$-$n$-sensitivity for any $n\in\bbn$. And in transitive systems, we prove that 
for any family $\calf$ with  $\{\mathbb{N}\}\subset \calf \subset \calf_{ps}$, broken $\calf$-$n$-sensitivity is equivalent to blockily thickly $n$-sensitivity. Because of this fact 
we choose $\calf_{ps}$, $\calf_{pubd}$ and $\calf_{inf}$ to consider  broken family sensitivity and obtain new results.
Moreover, we show some examples to distinguish different kinds of broken family sensitivity in both minimal and weakly mixng systems, which support
the view that the notion of broken family sensitivity is a natural one.

\medskip
We use essential sensitive tuples to 
give equivalent 
characterizations of these kinds of broken family sensitivity in transitive systems.
The following are the main results of this paper:

\begin{thmx}\label{thm:broken=}
	Let $(X,T)$ be a transitive system and $\calf=\calf_{ps}\ \text{or}\ \calf_{pubd}$.  Then $(X,T)$ is broken  $\calf$-$n$-sensitive if and only if there exists an essential $n$-sensitive tuple which is
	an $\calf$-recurrent point of $(X^n, T^{(n)})$.
\end{thmx}

\begin{thmx}\label{thm:inf}
	Let $(X,T)$ be a transitive system and $n\geq 2$.
	Then  $(X,T)$ is broken $\calf_{inf}$-$n$-sensitive
	if and only if there exists an essential $n$-sensitive tuple $(x_1,x_2,\dotsc,x_n)$
	in $(X,T)$ such that
	\[\limsup_{k\to\infty}\min_{1\leq i<j\leq n}d(T^kx_i,T^kx_j)>0. \]	
\end{thmx}

Using these characterizations we show that a minimal system $(X,T)$ is broken $\calf_{inf}$-sensitive
if and only if it is not an asymptotic extension of its maximal equicontinuous factor; is broken $\calf_{pubd}$-sensitive if and only if it is not a Banach proximal extension to its maximal equicontinuous factor; is broken $\calf_{ps}$-sensitive if and only if it is not a proximal extension to its maximal equicontinuous factor. 

The paper is organized as follows.
In Section $2$, we recall some definitions and
related results which will be used later. In Section $3$,
we define broken $\calf$-$n$-sensitivity, investigate
broken $\calf$-$n$-sensitivity for
$\calf=\calf_{inf}, \calf_{ps}\ \text{and}\ \calf_{pubd}$ in transitive systems and prove the main results, besides, some specific properties of them and examples in minimal systems are given in this section. In Section $4$, we consider 
broken $\calf$-$n$-sensitivity for
$\calf=\calf_{inf}, \calf_{ps}\ \text{and}\ \calf_{pubd}$ in weakly mixing systems and give some examples.

\section{Preliminaries}
\subsection{Topological dynamical systems}
In the article, sets of integers, nonnegative integers and
natural numbers are denoted by $\mathbb{Z},\ \mathbb{Z}_+$ and $\mathbb{N}$ respectively.
By a \emph{topological dynamical system}, we mean a pair $(X,T)$, where $X$
is a compact metric space with a  metric $d$ and $T\colon X\to X$ is a continuous surjective map.
Write $(X^n, T^{(n)})$ for the $n$-fold product system $(X \times\dotsb\times X, T \times\dotsb\times T)$, 
set $\Delta_n(X)=\{(x,\dotsc, x)\in X^n: x\in X\}$
and $\Delta^{(n)}(X)=\{(x_1,\dotsc, x_n)\in X^n: \exists 1\leq i\not=j\leq n \ \text{such that}\ x_i=x_j\}$.
A nonempty closed
invariant subset $Y\subseteq X$ defines naturally a subsystem $(Y, T)$ of $(X, T)$.
A system $(X,T)$ is called \emph{minimal} if it contains no proper subsystem.
Each point belonging to some minimal subsystem of $(X,T)$ is called a \emph{minimal point}.
\emph{The orbit of a point $x\in X$} is the set $Orb(x,T)=\{T^nx:\ n\in\Z_+\}$.

For $x\in X$ and $U,V\subset X$, put
$$N(x,U)=\{n\in\Z_+: T^nx\in U\}\ \text{and}\ N(U,V)=\{n\in\Z_+: U\cap T^{-n}V\neq\emptyset\}.$$
Recall that a dynamical system $(X,T)$ is called \emph{topologically transitive}
(or just \emph{transitive}) if for every two opene subsets $U,V$ of $X$,
the set $N(U,V)$ is infinite.  Any point with dense orbit is called a \emph{transitive point}.
Denote the set of all transitive points by $\tran(X,T)$.
It is well known that for a transitive system, $\tran(X,T)$ is a dense $G_\delta$ subset of $X$. A system $(X,T)$ is called \emph{weakly mixing} if the product system $(X^2, T^{(2)})$
is transitive.
We say that a transitive system $(X,T)$ is an $M$-system if the system has a dense set of minimal points;
a transitive system $(X,T)$ is an $E$-system if there exists a $T$-invariant probability measure $\mu$ on $X$ with $\mu(U)>0$ for any  
 opene set $U\subset X$.

We need some classical results.

\begin{lem}\cite[Proposition 2.3]{F67}\label{lem:weakly-mixing-n}
If $(X,T)$ is a weakly mixing system, then for any $n\geq 2$, the $n$-fold product system $(X^{n}, T^{(n)})$ is transitive.
\end{lem}

\begin{lem}\cite[Lemma 2.8]{Ethan-01}\label{lem: XY-minimal-points}
	Let $(X,T)$ and $(Y,S)$ be two dynamical systems with dense minimal points. Then the set of minimal points in $(X\times Y,T\times S)$ is dense.
\end{lem}

\subsection{Furstenberg family}\label{section-family}
For the set of non-negative
integers $\bbz_+$,
denote by $\calp=\calp(\bbz_+)$ the collection of all subsets of $\bbz_+$. A subset $\calf$ of $\calp$ is called a \emph{Furstenberg family} (or just \emph{family}) if it is hereditary upward, i.e., $F_1\subset F_2$ and $F_1\in\calf$ imply
$F_2\in\calf$. A family $\calf$ is called \emph{proper} if it is
neither empty nor all of $\calp$. If a proper family $\calf$ is closed under finite intersection, then $\calf$ is called a \emph{filter}.
Recall that a subset $F$ of $\mathbb{Z}_{+}$ is
\begin{enumerate}
	\item \emph{infinite} if it is an infinite subsets of $\bbz_+$, and
	denote by
	$\calf_{inf}$ the family of all infinite subsets of $\bbz_+$;
	\item \emph{thick} if it contains arbitrarily long blocks of consecutive integers, that is,
	for every $l\geq 1$ there is $n_l\in\N$ such that $\{n_l,n_l+1,\dotsc,n_l+l\}\subset F$, and denote by $\calf_{t}$ the family of all thick subsets of $\bbz_+$;
	\item \emph{syndetic} if it has bounded gaps, that is, there exists $N\in\N$ such that for every $k\in\N$ we have
	$\{k,k+1,\dotsc,k+N\}\cap F\neq\emptyset$, and denote by $\calf_{s}$ the family of all syndetic subsets of $\bbz_+$;
	\item \emph{piecewise syndetic} if
	there exist a thick set $A$ and a syndetic set $B$ such that $F=A\cap B$, and denote by $\calf_{ps}$ the family of all piecewise syndetic subsets of $\bbz_+$.	
\end{enumerate}

Let $F$ be a subset of $\bbz_+$, the upper Banach density of $F$ is
\[\overline{BD}(F)=\limsup_{n-m\to\infty}\frac{|F\cap\{m,m+1,\dotsc,n-1\}|}{n-m},\]
the lower Banach density of $F$ is
\[\underline{BD}(F)=\liminf_{n-m\to\infty}\frac{|F\cap\{m,m+1,\dotsc,n-1\}|}{n-m},\]
where $|\cdot|$ denote the cardinality of the set. A subset $F$ of $\mathbb{Z}_{+}$ is called \emph{positive upper Banach density} if $\overline{BD}(F)>0$, and denote by $\calf_{pubd}$ the family of all positive upper Banach density subsets of $\bbz_+$.
A subset $F$ of $\mathbb{Z}_{+}$ is called \emph{lower Banach density one} if $\underline{BD}(F)=1$, and denote by $\calf_{lbd1}$ the family of all lower Banach density one subsets of $\bbz_+$.

Let $\calf$ be a family and $(X,T)$ be a system. A point $x\in X$ is called 
\emph{$\calf$-recurrent point} if for any neighbourhood $U$ of $x$,  $N(x,U)=\{n\in\bbz_+:T^n x\in U\}\in\calf$.
Now we introduce a property of the orbit of a point along $F\in\calf=\calf_{ps}\ \text{or}\ \calf_{pubd}$ as follows.

\begin{lem}\label{cor:calf-recurrent}\cite[Proposition 5.1\ \text{and}\ Remark 5.7]{L12}
For any dynamical system $(X,T)$ and $x\in X$, $\overline{T^F(x)}=\overline{\{T^n x:n\in F\}}$ has
\begin{enumerate}
\item  an $\calf_{pubd}$-recurrent point, if $F\in \calf_{pubd}$.
\item  an $\calf_{s}$-recurrent point, if $F\in \calf_{ps}$.
\end{enumerate}
\end{lem}

In the next lemma we give an equivalent condition for a point to be the unique $\calf_{pubd}$-recurrent point of a system.

\begin{lem}\label{lem:only-recurrent}
	Let $(X,T)$ be a dynamical system and $x_0\in X$. Then $x_0$ is the unique $\calf_{pubd}$-recurrent point of $(X,T)$ if and only if for any $x\in X$ and neighbourhood $U$ of $x_0$, $N(x,U)\in \calf_{lbd1}$.
\end{lem}

\begin{proof}
	$(\Leftarrow)$ 
	Assume that there exists a $\calf_{pubd}$-recurrent point $x\in X$ with $x\not=x_0$, then for any neighbourhood $U$ of $x$ and $V$ of $x_0$, $N(x,U)\in \calf_{pubd}$ and $N(x,V)\in \calf_{lbd1}$, thus $N(x,U)\cap N(x,V)\not=\emptyset$. Since $U$ and $V$ are arbitrary, $x=x_0$.
	
	$(\Rightarrow)$ 
	Assume there exists $x\in X$ and a neighbourhood $V$ of $x_0$ such that 
	$N(x,V)\not\in \calf_{lbd1}$, 
	then $F:=N(x,X\setminus V)\in \calf_{pubd}$, by Lemma~\ref{cor:calf-recurrent}, $\overline{\{T^n x:n\in F\}}$ has a $\calf_{pubd}$-recurrent point. It is clear
	that $x_0\not\in \overline{\{T^n x:n\in F\}}$, which is a contradiction.
\end{proof}

\subsection{Sensitive tuples and regionally proximal tuples}\label{subsection-proximal}
 Let $(X,T)$ be a dynamical system and $n\geq 2$.
 We call \emph{a tuple $(x_1,x_2,\dotsc,x_n)\in X^n\setminus \Delta_n(X)$  sensitive},
 if for every opene subset $U$ of $X$, open neighborhoods $U_i$ of $x_i$ for $i=1,2,\dotsc,n$, there exist $y_1,y_2,\dotsc,y_n\in U$
 and $m\in\bbn$ such that
 \[
 T^m y_i\in U_i,\ \forall i=1,2,\dotsc,n.
 \]
 If in addition $x_i\neq x_j$ for $i\neq j$, we call $(x_1,x_2,\dotsc,x_n)$ \emph{an essential $n$-sensitive tuple}. Denote $S_n(X,T)$ as the collection of all $n$-sensitive tuples in $(X,T)$ and $S_n^{e}(X,T)$ the collection of all essential $n$-sensitive tuples in $(X,T)$. It is easy to see that $S_n(X,T)\cup \Delta_n(X)$ is closed and $T^{(n)}$-invariant.
 Fix $n\geq 2$, $(x_1,x_2,\dotsc, x_n)\in X^n$ and $y\in X$. 
 If for any neighborhoods $U_i$ of $x_i$ and $V$ of $y$, there exist $y_1, y_2,\dotsc, y_n\in V$ and $m\in \mathbb{N}$ such that
 $T^my_i\in U_i$, then we denote
 \emph{$y\in L(x_1,x_2,\dotsc, x_n)$}. Note that $L(x_1,x_2,\dotsc,x_n)$ is closed and $T$-invariant.  
 For transitive system $(X,T)$ and $x\in \tran(X,T)$, $(x_1,x_2,\dotsc,x_n)\in X^n\setminus \Delta_n(X)$ is a sensitive tuple if and only if 
 $x\in L(x_1,x_2,\dotsc,x_n)$.

 A pair  $(x_1,x_2)\in X\times X$ is said to be 
 \begin{enumerate}
 \item \emph{asymptotic} if $\lim_{m\to\infty}d(T^mx_1,T^mx_2)=0$;
\item \emph{Banach proximal} if for every $\eps>0$, $d(T^mx_1,T^mx_2)<\eps$ for all $m\in \bbz_+$ except a set of zero Banach density;
\item  \emph{proximal} if for any $\ep>0$, there exists a positive integer $m$ such that $d(T^mx_1,T^mx_2)<\ep$; 
 \item \emph{regionally proximal} if for each $\ep> 0$ and each open neighborhood $U_i$ of $x_i$, $i = 1, 2$,
 there are $x_i'\in U_i, i = 1, 2$, and $m\in \N$ with $d(T^mx_1', T^mx_2')<\ep$.
 \end{enumerate}
 
 Let $AS(X,T)$ (resp. $BP(X,T)$, $P(X,T), Q(X,T)$) denote the collection of all asymptotic pairs 
 (resp. Banach proximal pairs, proximal pairs, regionally proximal pairs) in $(X,T)$.
 Apparently, $AS(X,T)\subset BP(X,T)\subset P(X,T)\subset Q(X,T)$.
Note that $Q(X,T)$ is a reflexive symmetric $T\times T$-invariant closed relation, but is in general not transitive.
 However for minimal system $(X,T)$, $Q(X,T)$ is a closed invariant equivalence relation. Denote $Q(X,T)[x]=\{y\in X: (x,y)\in Q(X,T)\}$

 Generally, a tuple $(x_1,x_2,\dotsc,x_n)$ is called
 \emph{regionally $n$-proximal tuple} if for each $\ep> 0$ and each open neighborhood $U_i$ of $x_i$, $i = 1, \dotsc, n$, 
 there are $x_i'\in U_i, i = 1, \dotsc, n$,
 and $m\in \N$ with $\max_{1\leq i<j\leq n}d(T^mx_i', T^mx_j')<\ep$. 
 Let  $Q_n(X,T)$ denote the collection of all regionally $n$-proximal tuples in $(X,T)$.

For sensitive tuples and regionally proximal tuples, we have the following results.
\begin{lem}\label{lem:Sn-Qn}\cite[Theorem 3.4]{YZ08}
	Let $(X,T)$ be a dynamical system. Then
	\begin{enumerate}
		\item $S_n(X,T)\subset Q_n(X,T)$ for every $n\geq 2$;
		\item $S_n(X,T)=Q_n(X,T)\setminus \Delta_n(X)$  for every $n\geq 2$, provide that $(X,T)$ is minimal.
	\end{enumerate}
\end{lem}

\begin{thm}\label{thm:Q-Q-n}\cite[Corollary 6.9]{HLY11}
	Let $(X,T)$ be a minimal system and $n\geq 2$.
	If  $x_1,x_2,\dotsc,x_n\in X$ satisfy $(x_i,x_{i+1})\in Q(X,T)$ for $1\le i\le n-1$,
	then $(x_1,x_2,\dotsc,x_n)\in Q_n(X,T)$.
\end{thm}

\subsection{Maximal equicontinuous factor}
Let $(X,T)$ and $(Y,S)$ be two dynamical systems.
If there is a continuous surjection $\pi\colon X \to Y$ such that $\pi\circ T = S\circ \pi$, then we say that $\pi$ is a \emph{factor map}, $(Y,S)$ is a \emph{factor} of $(X,T)$ or $(X, T)$ is an \emph{extension} of $(Y,S)$.

For a factor map $\pi\colon (X,T)\to (Y,S)$, let
\[
R_\pi=\{(x_1,x_2)\in X^2 \colon \pi(x_1)=\pi(x_2)\}.
\]
Then $R_\pi$ is a $T\times T$-invariant closed equivalence relation on $X$ and $Y=X/R_\pi$.
In fact, there exists a one-to-one correspondence between the collection of factors of $(X,T)$ and the collection of $T\times T$-invariant closed  equivalence relations on $X$. Every topological dynamical system $(X,T)$ has a maximal equicontinuous factor $(X_{eq},T_{eq})$, that is $(X_{eq},T_{eq})$ is equicontinuous
and every equicontinuous factor of $(X,T)$ is also
a factor of $(X_{eq}, T_{eq})$.
There is a closed $T\times T$-invariant equivalence relation $S_{eq}$ on $(X,T)$, 
called the \emph{equicontinuous structure relation } such that $X/S_{eq}=X_{eq}$,
and $S_{eq}$ is the smallest closed $T\times T$-invariant equivalence relation containing $Q(X,T)$.

A factor map $\pi\colon (X, T) \to (Y, S)$ is called \emph{almost one-to-one} if $$\{x \in X \colon \pi^{-1}(\pi(x)) \textrm{ is a singleton} \}$$ is residual in $X$; 

A factor map $\pi\colon (X, T) \to (Y, S)$ is called \emph{an asymptotic extension} (resp. \emph{a Banach proximal extension}, \emph{a  proximal extension})
if $$AS(X,T)\supset R_\pi \ (\text{resp.}\  BP(X,T)\supset R_\pi, P(X,T)\supset R_\pi)$$

\section{Broken family sensitivity}
In this section we introduce broken $\calf$-$n$-sensitivity and prove the main results.
\begin{defn}
	Let $(X,T)$ be a dynamical system, $\calf$ a Furstenberg family and $n\geq 2$.
	We say that $(X,T)$ is \emph{broken  $\calf$-$n$-sensitive}
	if there exist $\delta>0$ and $F\in\calf$ with the property that  for every opene subset $U$ of $X$ and every $l\in\bbn$,
	there exists $x_1^{l},x_2^{l},\dotsc,x_n^{l}\in U$ and $m_{l}\in \bbz_+$ such that
	\[
	\min_{1\le i<j\le n}d(T^k x_i^{l}, T^k x_j^{l})> \delta,\ \forall k\in m_{l}+ F\cap[1,l].
	\]
\end{defn}
Following from the definitions we know that broken $\{\mathbb{N}\}$-$n$-sensitive is equivalent to
blockily thickly $n$-sensitive for any $n\in\bbn$.

\subsection{General properties of several broken family sensitivity}\label{subsection-general}
At the beginning of this section, we will
investigate broken $\calf$-$n$-sensitivity
when $\calf=\calf_{ps}\ \text{or}\ \calf_{pubd}$ and prove one of the main result. The study of these kinds of sensitivity are based on Lemma~\ref{cor:calf-recurrent}. For a system, the recurrence characterization of a point in this system is important. Researchers investigate various recurrence properties of a system and use them to classify transitive systems. Theorem~\ref{thm:broken=} give explicit
equivalent characterizations of these kinds of sensitivity with the recurrence properties and sensitive tuples.

\begin{proof}[Proof of Theorem~\ref{thm:broken=}]
$(\Leftarrow)$
Assume that $(x_1,\dotsc, x_n)$ is an
essential $n$-sensitive tuple and an $\calf$-recurrent point of $(X^n,T^{(n)})$.
Let $U_i$ be the neighbourhoods of $x_i$, for $i=1,2,\dotsc,n$ with
\[\min_{1\leq i<j\leq n} d(U_i,U_j)>\delta\ \text{for some}\ \delta>0.\]
Let $F=N((x_1,\dotsc, x_n),U_1\times \dotsb \times U_n)$. Then $F\in \calf$.
For any $l\in \bbn$, by the continuity of $T$ there exist
a neighbourhood $V_i^l$ of $x_i$ for $i=1,2,\dotsc,n$ such that for any $k\in F\cap[1,l]$, $T^k V_i^l\subset U_i$.
Since $(x_1,\dotsc, x_n)\in S_n^e(X,T)$,
for any opene subset $U$,
there exist $z_1^l\dotsc,z_n^l\in U$ and $m_l\in\bbn$ such that $T^{m_l} z_i^l\in V_i^l$, therefore
\[\min_{1\leq i<j\leq n}d(T^k z_i^l, T^k z_j^l)>\delta, k\in m_l+ F\cap[1,l]\]

$(\Rightarrow)$
Take $x\in \tran(X,T)$ and $U_m$ the neighbourhood of $x$ with
$\diam(U_m)<\frac{1}{m}$. By the definition of broken $\calf$-$n$-sensitive, there exist $\delta>0$ and $F\in\calf$ such that for any $m\in\bbn$,
there exist $x_1^{m},x_2^{m},\dotsc, x_n^{m}\in U_m$ and $p_m\in\bbn$
with the property that
\begin{equation}
d(T^k x_i^{m}, T^k x_j^{m})>\delta, k\in p_m+F\cap[1,m]
\end{equation}
Now we let $T^{p_m} x_i^{m}\rightarrow x_i$ when $m\to\infty$. It is clear that
$x\in L(x_1,x_2,\dotsc, x_n)$.
By the property of transitive point and $L(x_1,x_2,\dotsc, x_n)$, we have $X=L(x_1,x_2,\dotsc, x_n)$, that is $(x_1,x_2,\dotsc, x_n)$ is a sensitive tuple.
Since $T^{p_m} x_i^{m}\rightarrow x_i$ when $m\to\infty$, by (3.1) we have
\[d(T^k x_i, T^k x_j)\geq\delta,\ \forall k\in F.\]
By Lemma~\ref{cor:calf-recurrent}, there is an $\calf$-recurrent point $(y_1,y_2,\dotsc, y_n)\in \overline{(T^{(n)})^F(x_1,x_2,\dotsc, x_n)}$.
(Note that if  $\calf=\calf_{ps}$, $(y_1,y_2,\dotsc, y_n)$ can be an $\calf_{s}$-recurrent point.)
Since
$S_n(X,T)\cup \Delta_n(X)$ is a $T^{(n)}$-invariant closed set,
$\overline{(T^{(n)})^F(x_1,x_2,\dotsc, x_n)}\subset S_n(X,T)\cup \Delta_n(X)$.
It is clear that $\overline{T^{(n)F}(x_1,x_2,\dotsc, x_n)}\cap \Delta^{(n)}(X)=\emptyset$, thus
$(y_1,y_2,\dotsc, y_n)$ is
an essential $n$-sensitive tuple.
\end{proof}

We know that every dynamical system has a maximal equicontinuous factor (\cite{YZ08}). The following theorem characterizes how broken $\calf$-$n$-sensitivity ($\calf=\calf_{ps}\ \text{or}\ \calf_{pubd}$) relates to the fiber of its maximal equicontinuous factor.

\begin{thm}
	Let $(X,T)$ be a minimal system and $\pi:(X,T)\rightarrow (X_{eq},T_{eq})$ be the factor map to its maximal equicontinuous factor. 
Then $(X,T)$ is
	broken $\calf$-$n$-sensitive ($\calf=\calf_{ps}\ \text{or}\ \calf_{pubd}$) if and only if there exist $y\in X_{eq}$ and distinct points
	$x_1,\dotsc,x_n\in \pi^{-1}(y)$ such that
	$(x_1,\dotsc,x_n)$ is an $\calf$-recurrent point of $(X^n,T^{(n)})$.
\end{thm}
\begin{proof}
$(\Rightarrow)$	
By Theorem~\ref{thm:broken=}, there exists an essential $n$-sensitive tuple $(x_1,\dotsc,x_n)$ 
which is an $\calf$-recurrent point of $(X^n,T^{(n)})$.
By Lemma~\ref{lem:Sn-Qn}, $S_n^e(X,T)\subset Q_n(X,T)$ and we know that in minimal systems $R_{\pi}=Q(X,T)$. 
Thus there exists $y\in X_{eq}$ such that $x_1,\dotsc,x_n\in \pi^{-1}(y)$.

$(\Leftarrow)$	
Since in minimal systems $R_{\pi}=Q(X,T)$,
$(x_1,x_i)\in Q(X,T)$ for $i=2,\dotsc,n$.	
By Theorem~\ref{thm:Q-Q-n}, 	
$(x_1,\dotsc,x_n)\in Q_n(X,T)$.
By Lemma~\ref{lem:Sn-Qn}, $S_n^e(X,T)= Q_n(X,T)\setminus\Delta^{(n)}(X)$, 
thus $(x_1,\dotsc,x_n)$ is an essential $n$-sensitive tuple which is an $\calf$-recurrent point of $(X^n,T^{(n)})$.
By Theorem~\ref{thm:broken=},
$(X,T)$ is broken $\calf$-$n$-sensitive.
\end{proof}

\subsection{Specific properties of broken \texorpdfstring{$\calf_{ps}$-$n$}{ps-n}-sensitivity}
In Subsection~\ref{subsection-general} we investigate general properties of
broken $\calf_{ps}$-$n$-sensitivity
and give explicit equivalent characterizations.
In this subsection we will further give specific properties of broken $\calf_{ps}$-$n$-sensitivity.

\begin{thm}\label{main1}
	Let $(X,T)$ be a transitive system and $n\geq 2$.
	Then the following statements are equivalent:
	\begin{enumerate}
		\item $(X,T)$ is broken  $\{\bbn\}$-$n$-sensitive;
		\item $(X,T)$ is broken $\calf_{ps}$-$n$-sensitive;
		\item there exists an essential $n$-sensitive tuple $(x_1,x_2,\dotsc,x_n)$ which is a distal tuple;
		\item there exists an essential $n$-sensitive tuple which is an $\calf_s$-recurrent point of $(X^n, T^{(n)})$;
		\item there exists an essential $n$-sensitive tuple which is an $\calf_{ps}$-recurrent point of $(X^n, T^{(n)})$.
	\end{enumerate}
\end{thm}

\begin{proof}
	$(1)\Rightarrow(2)$ follow from definitions, $(2)\Leftrightarrow(4)\Leftrightarrow(5)$ follow from
	the proof of Theorem~\ref{thm:broken=}.
	
	$(4)\Rightarrow(1)$
	Let $M=\overline{\orb((x_1,x_2,\dotsc,x_n),T^{(n)})}$. Since $(x_1,x_2,\dotsc,x_n)$ is a minimal point and $\Delta^{(n)}(X)$ is a $T^{(n)}$-invariant closed set, $M\cap \Delta^{(n)}(X)=\emptyset$. Take an open neighbourhood $W$ of $M$ 
	such that $\overline{W}\cap \Delta^{(n)}(X)=\emptyset$.
	Then there exists $\delta>0$ such that
	for any $(y_1,y_2,\dotsc y_n)\in W$,
	$\min_{1\leq i<j\leq n}d(y_i,y_j)>\delta$.
	Fix an opene set $U$ of $X$ and $L\in\bbn$.
	Since $M$ is $T^{(n)}$-invariant, $\bigcap_{l=0}^{L}(T^{(n)})^{-l}W$ is an open neighbourhood of $M$. Thus for any $i=1,2,\dotsc,n$, there exists a neighbourhood $U_i$ of $x_i$ such that
	\[U_1\times U_2\times \dotsb \times U_n\subset \bigcap_{l=0}^{L}(T^{(n)})^{-l}W.\]
	Since $(x_1,x_2,\dotsc,x_n)$ is an $n$-sensitive tuple, there exist $z_1,z_2,\dotsc,z_n\in U$ and $k\in\bbn$ such that $T^k z_i\in U_i$, $i=1,2,\dotsc,n$.
	By the construction of $\delta$ and $W$,
	\[\min_{1\leq i<j\leq n}d(T^{k+l}z_i, T^{k+l}z_j)>\delta, l=0,1,2,\dotsc,L.\]
	So $(X,T)$ is broken $\{\N\}$-$n$-sensitive.
	
	$(1)\Rightarrow(3)$
	Fix $x\in \tran(X)$ and let $U_m$ be the neighborhood of $x$ such that $\diam(U_m)<\frac{1}{m}$. By the definition, there exist $\delta>0$, $\{x_1^m, x_2^m,\cdots,x_
	n^m\}\subset U_m$ and $k_m\in \bbn$ such that  
	\[\min_{1\leq i\neq j \leq n}d(T^{k_m+p}{x_i^m}, T^{k_m+p}{x_j^m})>\delta, \forall 0 \leq p \leq m-1.\]
	Denote $x_i$ be the limit of $T^{k_m}{x_i^m}$ as $m\to\infty$, for $i=1,2,\cdots,n$.
	It is clear that $(x_1,x_2,\cdots,x_n)$ is a distal tuple and $x\in L(x_1,x_2,\cdots,x_n)$. Since $L(x_1,x_2,\cdots,x_n)$ is closed and $T$-invariant, $(x_1,x_2,\cdots,x_n)\in S_n^e(X,T)$.
	
	$(3)\Rightarrow(4)$
	Since $S_n(X,T)\cup \Delta_n$ is a $T^{(n)}$-invariant closed set, 
	$\overline{\orb((x_1,x_2,\cdots,x_n),T^{(n)})}\subset S_n(X,T)\cup \Delta_n$.
	Since $(x_1,x_2,\cdots,x_n)$ is a distal tuple, $\overline{\orb((x_1,x_2,\cdots,x_n),T^{(n)}}\cap \Delta^{(n)}=\emptyset$. Therefore any minimal point in \[\overline{\orb((x_1,x_2,\cdots,x_n),T^{(n)})}\] is the essential $n$-sensitive tuple we need.
\end{proof}

\begin{rem}\label{rem:calf==}
By Theorem~\ref{main1}, for a family $\calf$ with  $\{\mathbb{N}\}\subset \calf \subset \calf_{ps}$, broken $\calf$-$n$-sensitivity,
	broken $\calf_{ps}$-$n$-sensitivity, 
	broken $\{\bbn\}$-$n$-sensitivity 
	and blockily thickly-$n$-sensitivity
	are equivalent in transitive systems.
\end{rem}

Following from Remark~\ref{rem:calf==}, we know that the result of \cite[Theorem A]{Y18} is established also for broken $\calf_{ps}$-sensitivity, which is the following corollary.

\begin{cor}\label{cor:ps-proximal-exten}
	Let $(X,T)$ be a minimal system. Then $(X,T)$
	is broken $\calf_{ps}$-sensitive if and only if it is not a proximal extension to its maximal equicontinuous factor. 
\end{cor}

A system $(X,T)$ is called \emph{topological ergodic} if for any opene subsets $U$ and $V$ of $X$, $N(U,V)=\{n\in\bbn: T^{-n}U\cap V\not=\emptyset \}$ is a syndetic set.
Now using Theorem~\ref{main1} we prove that if $(X,T)$ is a topological ergodic system containing $n$ different minimal subsets $M_1, M_2, \dotsc, M_n$, then it is  broken $\calf_{ps}$-$n$-sensitive.

\begin{lem}\label{lem:topo-ergodic}\cite[Theorem 7.5]{YZ08}
	Let $(X,T)$ be a topological ergodic system and $n\geq 2$.
	If there exist pairwise disjoint minimal subsets $M_1, M_2, \dotsc, M_n$ of $(X,T)$, then there exists an essential $n$-sensitive tuple $(y_1,\dotsc,y_n)$ which is a minimal point of $(X^n,T^{(n)})$.
\end{lem}

\begin{prop}\label{coro:dense-minimal-point}
	Let $n\geq 2$ and $(X,T)$ be a topological ergodic system containing $n$ different minimal subsets $M_1, M_2,\dotsc, M_n$. Then $(X,T)$ is broken $\calf_{ps}$-$n$-sensitive.
\end{prop}
\begin{proof}
	By Lemma~\ref{lem:topo-ergodic}, there exists an essential $n$-sensitive tuple $(x_1,\dotsc,x_n)$ which is a minimal point of $(X^n,T^{(n)})$.
	Thus by Theorem~\ref{main1}, $(X,T)$ is broken $\calf_{ps}$-$n$-sensitive.
\end{proof}

Since any non-minimal $M$-system is topological ergodic containing $n$ different minimal subsets for any $n\in \N$, we have the following corollary.
\begin{cor}
	Let $(X,T)$ be an $M$-system which is not minimal. Then for every $n\geq 2$, $(X,T)$ is broken $\calf_{ps}$-$n$-sensitive.
\end{cor}

\subsection{Specific properties of broken \texorpdfstring{$\calf_{pubd}$-$n$}{pubd-n-}-sensitivity}
In Subsection~\ref{subsection-general} we investigate general properties of
broken $\calf_{pubd}$-$n$-sensitivity
and give explicit equivalent characterizations.
In this subsection we will further give specific properties of broken $\calf_{pubd}$-$n$-sensitivity.

Let $(X,T)$ be a dynamical system and $M(X)$ be the set of Borel probability measures on $X$. If for any $B\in\mathcal{B}(X)$, $\mu(T^{-1}B)=\mu(B)$, where $\mathcal{B}(X)$ is the Borel $\sigma$-algebra of $X$, we call $\mu$ a $T$-invariant Borel probability measure and denote by $M(X,T)$ the set of all $T$-invariant Borel probability measures on $X$.
The support of a measure $\mu\in M(X)$, denoted by $\supp(\mu)$, is defined as follow:
\[\supp(\mu)=\{x\in X:\ \text{for any neighbourhood}\ U\ \text{of}\ x, \mu(U)>0
\},\]
And the support of a system $(X,T)$, denoted by $\supp(X,T)$, is defined as follow:
\[\supp(X,T)=\overline{\bigcup\{\supp(\mu): \mu\in M(X,T)}\},\]

\begin{lem}\label{lem:supp-pld}
	Let $(X,T)$ be a dynamical system.
	Then
	\begin{enumerate}
		\item\cite[Proposition 3.10]{LT14} $\supp(X,T)=\overline{\{x\in X: x\ \text{is a }\calf_{pubd}\text{-recurrent point}\}}$;
		\item\cite[ Proposition $3.12$]{LT14} 
		$\supp(X^n,T^{(n)})=\supp(X,T)\times\dotsb\times\supp(X,T)$.
	\end{enumerate}
\end{lem}

Following from Lemma~\ref{lem:supp-pld}(1) and Theorem~\ref{thm:broken=}, we have the following corollary.
\begin{cor}\label{cor:pubd-supp}
	Let $(X,T)$ be a transitive system.
	If $(X,T)$ is broken $\calf_{pubd}$-n-sensitive then $\supp(X^n,T^{(n)})\cap S_n^e(X,T)\not=\emptyset$.	
\end{cor}

A dynamical system $(X,T)$  is called  \textit{mean equicontinuous} if for any $\ep>0$ there is $\delta>0$ such that if $x,y\in X$ with $d(x,y)<\delta$ then  $\limsup_{m\to+\infty}\frac{1}{m}\sum_{k=0}^{m-1}d(T^kx,T^ky)<\ep$.
We refer the readers to \cite{LTY15,LYY21} for more details about mean equicontinuous system.

\begin{lem}\cite[Theorem 4.3]{QZ18}\label{lem: Q-P-BP}
Let $(X,T)$ be a dynamical system. Then $(X,T)$ is mean equicontinuous if and only if $Q(X,T)=P(X,T)=BP(X,T)$.
\end{lem}

\begin{prop}\label{prop:pubd-mean}
Let $(X,T)$ be a minimal system. Then the follwing statements are equivalent:
\begin{enumerate}
\item $(X,T)$ is broken $\calf_{pubd}$-sensitive; 
\item $(X,T)$ is not mean equicontinuous;
\item $(X,T)$ is not a Banach proximal extension to its maximal equicontinuous factor.
\end{enumerate}
\end{prop}
\begin{proof}
Let $\pi:(X,T)\to(X_{eq},T_{eq})$ be the factor map to its maximal equicontinuous factor.

$(1)\Rightarrow(2)$
Assume that $(X,T)$ is mean equicontinuous. By Lemma~\ref{lem: Q-P-BP}, $Q(X,T)=P(X,T)=BP(X,T)$. By Lemma~\ref{lem:Sn-Qn}, $S_2(X,T)\subset Q(X,T)$. Since $(X,T)$ is broken $\calf_{pubd}$-sensitive, by Theorem~\ref{thm:broken=} there is a sensitive pair $(x,y)$ 
which is an $\calf_{pubd}$-recurrent point of $X\times X$. Then 
$(x,y)\in S_2(X,T)\subset Q(X,T)=BP(X,T)$, which means for any $\eps>0$, $d(T^mx,T^my)<\eps$ for all $m\in\bbn$ except a set of zero Banach density. Let $\delta=\frac{d(x,y)}{3}$. Since $(x,y)$ is an $\calf_{pubd}$-recurrent point, there exists $F\in\calf_{pubd}$ such that for any $n\in F$, $(T^nx, T^ny)\in B(x,\delta)\times B(y,\delta)$, thus $d(T^nx, T^ny)>\delta$. This is a contradiction.

$(2)\Rightarrow(3)$
Assume that $(X,T)$ is not mean equicontinuous. By Lemma~\ref{lem: Q-P-BP}, $Q(X,T)\setminus BP(X,T)\not=\emptyset$. 
Since $R_{\pi}=Q(X,T)$ in minimal systems,
there exist $x_1,x_2\in X$ 
with $\pi(x_1)=\pi(x_2)$ 
such that $(x_1,x_2)\not\in BP(X,T)$, 
which implies that $(X,T)$ is not a Banach proximal extension to its maximal equicontinuous factor.

$(3)\Rightarrow(1)$
Assume that $(X,T)$ is not a Banach proximal extension to its maximal equicontinuous factor. Then $R_{\pi}\setminus BP(X,T)\not=\emptyset$. Since $R_{\pi}=Q(X,T)$ in minimal systems,
there exists $(x_1,x_2)\in Q(X,T)\setminus BP(X,T)$.
By Lemma~\ref{lem:Sn-Qn} $(x_1,x_2)\in S_2^e(X,T)$, and there exists $\delta>0$ and $F\in\calf_{pubd}$  such that
\[d(T^k x_1, T^k x_2)\geq\delta,\ \forall k\in F.\]
By Corollary~\ref{cor:calf-recurrent}, there exists an $\calf_{pubd}$-recurrent point $(y_1,y_2)\in \overline{(T^{(2)})^F(x_1,x_2)}$.
Since
$S_2(X,T)\cup \Delta_2(X)$ is a $T^{(2)}$-invariant closed set,
$\overline{(T^{(2)})^F(x_1,x_2)}\subset S_2(X,T)\cup \Delta_2(X)$.
It is clear that $\overline{(T^{(2)})^F(x_1,x_2)}\cap \Delta_2(X)=\emptyset$, thus
$(y_1,y_2)$ is an essential sensitive tuple. 
By Theorem~\ref{thm:broken=}, $(X,T)$ is broken $\calf_{pubd}$-sensitive.
\end{proof}

\subsection{Broken \texorpdfstring{$\calf_{inf}$-$n$}{inf-n}-sensitivity}

In this subsection we will discuss broken $\calf_{inf}$-$n$-sensitive and prove Theorem~\ref{thm:inf}. Besides, we give a specific property of broken $\calf_{inf}$-sensitivity.

\begin{proof}[Proof of Theorem~\ref{thm:inf}]
	$(\Rightarrow)$
	Fix $x\in \tran(X)$ and let $U_m$ be the neighborhood of $x$ such that $\diam(U_m)<\frac{1}{m}$.
	By the definition of broken
	$\calf_{inf}$-$n$-sensitive,
	there exist $\delta>0$ and $F=\{n_1, n_2,\dotsc\}\in\calf_{inf}$ with
	$n_1<n_2<\dotsc$, take $l_m\in \bbn$ such that $F\cap[1,l_m]=\{n_1, \dotsc, n_m\}$,
	then there exist $x_1^m,x_2^m,\dotsc,x_n^m\in U_m$ and $p_m\in \bbn$ such that
	\[
	\min_{1\le i<j\le n}d(T^k x_i^m, T^k x_j^m)\geq \delta,\ \forall k\in p_m+ F\cap[1,l_m].
	\]
	Then we have
	\[\min_{1\le i<j\le n}d(T^k x_i^m, T^k x_j^m)\geq \delta,\ \forall k\in \{p_m+n_{1},\dotsc,p_m+n_{m}\}.\]
	That is
	\begin{equation}
	\min_{1\le i<j\le n}d(T^k( T^{p_m}x_i^m), T^k (T^{p_m}x_j^m))\geq \delta,\ \forall k\in \{n_{1},\dotsc,n_{m}\}.
    \end{equation}
	Denote $x_i$  the limit of $T^{p_m}{x_i^m}$ as $m\to\infty$, for $i=1,2,\dotsc,n$.
	Then for any $k\in F=\{n_1, n_2,\dotsc\}$
	\[\min_{1\leq i<j\leq n}d(T^{k}x_i,T^{k}x_j)\geq \delta>0. \]
	And therefore
	\[\limsup_{k\to\infty}\min_{1\leq i<j\leq n}d(T^kx_i,T^kx_j)>0. \]
	
	It is clear that $x\in L(x_1,x_2,\dotsc,x_n)$. Since $x$ is a transitive point, by the property of $L(x_1,x_2,\dotsc,x_n)$, $X=L(x_1,x_2,\dotsc,x_n)$ and
	we have $(x_1,x_2,\dotsc,x_n)\in S_n^e(X,T)$.
	
	$(\Leftarrow)$
	Since there exists a strictly increasing subsequence $\{k_m\}$ of $\bbn$ such that
	\[3\delta=\lim_{m\to\infty}\min_{1\leq i<j\leq n}d(T^{k_m}x_i,T^{k_m}x_j)>0. \]
    Without loss of generality, we assume that for any $m\in \bbn$,
    $$\min_{1\leq i<j\leq n}d(T^{k_m}x_i,T^{k_m}x_j)>2\delta.$$
	By continuity,
	there exist neighborhoods $U_i^m$ of $x_i$ such that
	$$\min_{1\le i<j\le n}d(T^{k}U_i^m,T^{k}U_j^m)>2\delta, \ \text{for any}\ k\in\{k_1,\dotsc,k_m\}.$$

     Let $F=\{k_1, k_2,\dotsc\}\in\calf_{inf}$. For any opene set $U$ and $l\in \bbn$, there exists $m=m(l)\in \N$ such
     that $F\cap [1,l]=\{k_1, k_2,\dotsc,k_m\}$.
	Since $(x_1,x_2,\dotsc,x_n)$ is an $n$-sensitive tuple, there exist
	$y_1^l,\dotsc,y_n^l\in U$ and $q_l$
	such that $T^{q_l}y_i^l\in U_i^{m}$ and
	thus
	\[\min_{1\le i<j\le n}d(T^k(T^{q_l}y_i^l),T^k(T^{q_l}y_j^l))>\delta,\ \text{for any}\ k\in\{k_1,\dotsc,k_m\}.\]
	That is
	\[\min_{1\le i<j\le n}d(T^ky_i^l,T^ky_j^l)>\delta,\ \text{for any}\ k\in q_l+F\cap[1,l].\]
Thus $(X,T)$ is	broken
	$\calf_{inf}$-$n$-sensitive.
\end{proof}

For minimal systems, maximal equicontinuous factor is important for studying some dynamical properties, especially for sensitivity. Next proposition is a specific property of broken $\calf_{inf}$-sensitivity.

\begin{prop}\label{cor: inf}
	Let $(X,T)$ be a minimal system.
	Then  $(X,T)$ is broken $\calf_{inf}$-sensitive
	if and only if it is not an asymptotic extension of its maximal equicontinuous factor.
\end{prop}
\begin{proof}
$(\Rightarrow)$
Since $(X,T)$ is broken $\calf_{inf}$-sensitive,
by Theorem~\ref{thm:inf}, there exists a sensitive pair $(x_1,x_2)$ such that
\[\limsup_{k\to\infty}d(T^kx_1,T^kx_2)>0. \]
By Lemma~\ref{lem:Sn-Qn}, $S_2(X,T)\subset Q(X,T)$. Let $\pi: (X,T)\to (X_{eq}, T_{eq})$ be the factor map to its maximal equicontinuous factor. Then we have  $Q(X,T)\subset R_{\pi}$. So we have 
$\pi(x_1)=\pi(x_2)$ such that
\[\limsup_{k\to\infty}d(T^kx_1,T^kx_2)>0, \]
i.e. $R_{\pi}\not\subset AS(X,T)$,
which implies that $\pi$ is not an asymptotic extension.

$(\Leftarrow)$
Let $\pi: (X,T)\to (X_{eq}, T_{eq})$ be a factor map to its maximal equicontinuous factor. Then $R_{\pi}\setminus AS(X,T)\not=\emptyset$, i.e.
there exist $x_1, x_2$ with
$\pi(x_1)=\pi(x_2)$ such that
\[\limsup_{k\to\infty}d(T^kx_1,T^kx_2)>0, \]
Since $(X,T)$ is minimal, we have  $R_{\pi}=Q(X,T)$.
By Lemma~\ref{lem:Sn-Qn}, $S_2(X,T)=Q(X,T)\setminus\Delta_{2}(X)$.
Thus $(x_1,x_2)\in
R_{\pi}\setminus\Delta_{2}(X)= Q(X,T)\setminus\Delta_{2}(X)=S_2(X,T)$.
By Theorem~\ref{thm:inf} $(X,T)$ is broken $\calf_{inf}$-sensitive.
\end{proof}

\subsection{Examples for broken family sensitivity in minimal systems}\label{subsection-examples}
The following implications follow from the definitions, we show that all the implications in the diagram are strict in minimal systems. (sen. is short for sensitivity)
\begin{equation*}
\begin{array}{cccc}
\ \ \text{broken $\calf_{ps}$-$n$-sen.} \   \ \   \Rightarrow &  \ \ \ \text{broken $\calf_{pubd}$-$n$-sen.} \   \ \  \Rightarrow &  \text{broken $\calf_{inf}$-$n$-sen.}\\
\Uparrow & \Uparrow & \Uparrow  \\
\text{broken $\calf_{ps}$-$(n+1)$-sen.} \Rightarrow & \text{broken $\calf_{pubd}$-$(n+1)$-sen.} \Rightarrow &  \text{broken $\calf_{inf}$-$(n+1)$-sen.}
\end{array}
\end{equation*}

\begin{exam}
	[broken $\calf_{pubd}$-sensitivity $\not\Rightarrow$ broken $\calf_{ps}$-sensitivity]
	In \cite[Example $10.3$]{Downarowicz2005},
	The authors constructed a Toeplitz flow $(X,T)$ which is a minimal system having two ergodic
	measures. Furthermore, it is an almost one to one extension of odometers, which is its maximal equicontinuous factor.
	Since every minimal mean equicontinuous system is uniquely ergodic (see \cite{Fomin51}), $(X,T)$ is not
	mean equicontinuous. By proposition \ref{prop:pubd-mean},
	$(X,T)$ is broken $\calf_{pubd}$-sensitive. By Corollary~\ref{cor:ps-proximal-exten},  a minimal system is not broken $\calf_{ps}$-sensitive if and only if it is a proximal extension to its maximal equicontinuous factor. Since every almost one to one extension is proximal, 
	the Toeplitz flow $(X,T)$ is not broken $\calf_{ps}$-sensitive.
\end{exam}

\begin{exam}
	[broken $\calf_{inf}$-sensitivity $\not\Rightarrow$ broken $\calf_{pubd}$-sensitivity]
	In \cite[Theorem 3.1]{Downarowicz2016},
	the authors reveal that there exists a minimal mean equicontinuous system $(X,T)$ which is not an almost one to one extension to its maximal equicontinuous factor. Since every asymptotic extension is almost one to one,  
	$(X,T)$ is not an asymptotic extension of its maximal equicontinuous factor.
	By Corollary~\ref{cor: inf} and Proposition~\ref{prop:pubd-mean}, $(X,T)$ is broken $\calf_{inf}$-sensitive but not broken $\calf_{pubd}$-sensitive.
\end{exam}

\begin{exam}
	[broken $\calf_{ps}$-$n$-sensitivity $\not\Rightarrow$ broken $\calf_{pubd}$-$(n+1)$-sensitivity]
	In \cite[Example $2$ of Subsection $6.2$]{MShao07}, for every $n\geq 2$,
	the authors constructed an example which is called $n$-Morse system. We will show that $n$-Morse system is
	broken $\calf_{ps}$-$n$-sensitive but not broken $\calf_{pubd}$-$(n+1)$-sensitive.

	Define the substitution map $\tau$ on $\{0,1\dots,n-1\}$ such that
	$\tau(0) = 01\ldots (n-1),$ $\tau(1) =12\ldots 0,\dots,$ $\tau(n-1) =(n-1)0\ldots (n-2)$. By concatenating, this map can act on any finite word
	$w = w_0w_1\dots w_{l-1}$ in
	$\{0, 1,\dots,n-1\}^l$ by $ \tau(w) = \tau(w_0)\tau(w_1) \ldots \tau(w_{l-1})$.
	Let $\Sigma_n=\{0,1,\ldots,n-1\}^\Z$ and $ X\subset\Sigma_n$ be the set of bi-infinite $n$-ary sequences $x$ in $\Sigma_n$ such that
	any finite word of $x$ is a subword of $\tau^k(0)$ for some $k\in \N$. Obviously $X$ is
	closed in $\Sigma_n$ and is invariant under the  shift map $\sigma$. We still denote by $\sigma$ the shift map restricted on $X$ and call $(X,\sigma)$ \emph{$n$-Morse system}. It is well known that $(X, \sigma)$ is minimal and has the following structure:
	$\pi_1: (X,\sigma)\ra (Y,S)$ and $\pi_2: (Y,S)\ra (X_{eq},\sigma_{eq})$, where $\pi_1$ is an $n$-to-one distal extension and $\pi_2$ is
	an asymptotic extension.
	
	$(1)$ $n$-Morse system is
	broken $\calf_{ps}$-$n$-sensitive.
	Since $\pi_1: (X,\sigma)\to (Y,S)$ is  an $n$-to-one distal extension, there exists a minimal point
	$$(x_1,\dotsc,x_n)\in R_{\pi_1}^n(X,\sigma)\setminus\Delta^{(n)}(X).$$
	Since $(x_i,x_j)\in R_{\pi_2\circ\pi_1}(X,\sigma)$
	for any $1\le i, j\le n$ and $R_{\pi_2\circ\pi_1}(X,\sigma)=Q(X,\sigma)$, by Theorem~\ref
	{thm:Q-Q-n} $(x_1,\dotsc,x_n)\in Q_n(X,\sigma)$, and thus $(x_1,\dotsc,x_n)\in S_n^e(X,\sigma)$ by Lemma~\ref{lem:Sn-Qn}.
	By Theorem~\ref{main1}, $n$-Morse system $(X,\sigma)$ is broken $\calf_{ps}$-$n$-sensitive.
	
	$(2)$
	$n$-Morse system is
	not broken $\calf_{pubd}$-$(n+1)$-sensitive.
	Assume that $(X,\sigma)$ is broken $\calf_{pubd}$-$(n+1)$-sensitive, then by Theorem~\ref{thm:broken=} there exists
	an essential $(n+1)$-sensitive tuple $(x_1,\dotsc,x_{n+1})$ which is an $\calf_{pubd}$-recurrent point of $(X^{n+1},\sigma^{(n+1)})$, and thus $(x_1,\dotsc,x_{n+1})\in Q_{n+1}(X,\sigma)$ by Lemma~\ref{lem:Sn-Qn}. 
	Then $\pi_2\circ\pi_1(x_i)=\pi_2\circ\pi_1(x_j)$ for any $1\le i, j\le n+1$ 
	as $R_{\pi_2\circ\pi_1}(X,\sigma)=Q(X,\sigma)$. Since $\pi_1$ is $n$-to-one,  there exist some $i\not= j\in\{1,\dotsc,n+1\}$ such that $\pi_1(x_i)\not=\pi_1(x_j)$. Without loss of generality we assume that $\pi_1(x_1)\not=\pi_1(x_2)$.
	Then we have $(\pi_1(x_1),\pi_1(x_2))$ is an $\calf_{pubd}$-recurrent point of $(Y^2,S^{(2)})$ as
	$(x_1,\dotsc,x_{n+1})$ is an $\calf_{pubd}$-recurrent point of $(X^{n+1},\sigma^{(n+1)})$.
	This leads to a contradiction since  $\pi_2\circ\pi_1(x_1)=\pi_2\circ\pi_1(x_2)$ and $\pi_2$
	is an asymptotic extension.
\end{exam}

\begin{exam}[broken $\calf_{inf}$-$2$-sensitivity $\not\Rightarrow$ broken $\calf_{inf}$-$3$-sensitivity]
	For Morse system $(X,\sigma)$, we show that $(X,\sigma)$ is not broken $\calf_{inf}$-$3$-sensitive, then we know that broken $\calf_{inf}$-$2$-sensitive $\not\Rightarrow$ broken $\calf_{inf}$-$3$-sensitive.
	
	The Morse sequence $\omega(n)$:
	$0110100110010110\dotsc $
	can be described by the following algorithms.
	$$\omega(0)=0, \omega(2n)=\omega(n), \omega(2n+1)=1-\omega(n) (n\in \N).$$
	Considering $\omega$ as an element of $\Sigma_2=\{0,1\}^{\Z}$ where $\omega(-n)=\omega(n-1)$, then $\omega \in (X,\sigma)$.
	
	For $\xi \in \Sigma_2$, define
	the homeomorphism $\varphi:\xi \ra \overline{\xi}$
	where $\overline{\xi}(n)=\overline{\xi(n)}$ ($\overline{0}=1, \overline{1}= 0$).
	Then $\varphi(x)\in X$ for any $x\in X$.
	Define $\eta \in \Sigma_2$  by $\eta(n)=\omega(n)$ for $n\geq 0$ and
	$\eta(n)=\overline{\omega(n)}$ for $n< 0$,  then we have $\eta \in X$.
	
	From \cite{Glasner1989}, we
	note that for any $n\in \bbz$, 
	$Q(X,\sigma)[\sigma^n \omega]=\{\sigma^n \omega, \sigma^n \eta, \overline{\sigma^n \omega}, \overline{\sigma^n \eta}  \}$, 
	and for any $x\in X\setminus \{\sigma^n \omega, \sigma^n \eta, \overline{\sigma^n \omega}, \overline{\sigma^n \eta} : n\in \bbz \}$,
	$Q(X,\sigma)[x]=\{x,\overline{x} \}$. It is easy to see that $(\overline{\sigma^n \omega}, \overline{\sigma^n \eta})$ and
	$(\sigma^n \omega, \sigma^n \eta)$ are both asymptotic pairs of $(X,\sigma)$, so there do not exist an essential $3$-sensitive tuple $(x_1,x_2,x_3)$
	in $(X,\sigma)$ such that
	$$\limsup_{k\to\infty}\min_{1\leq i<j\leq 3}d(\sigma^kx_i,\sigma^kx_j)>0.$$
	By theorem \ref{thm:inf}, 	$(X,\sigma)$ is not broken $\calf_{inf}$-$3$-sensitive.
\end{exam}

\section{Broken family sensitivity in weakly mixing systems}
In this section, we devote to
investigate broken $\calf$-$n$-sensitive 
for
$\calf=\calf_{inf}, \calf_{ps}\ \text{and}\ \calf_{pubd}$
in weakly mixing systems. We will give some 
equivalent characterizations for them. Examples of these kinds of broken family sensitivity in weakly mixing systems are also provided in this section.

\subsection{Broken family sensitivity in weakly mixing systems}\label{section-weakmixing}

\begin{lem}\label{lem:Sn-Wm}\cite[Theorem$3.2$]{YZ08}
	Let $(X,T)$ be a dynamical system. Then the following statements are equivalent:
	\begin{enumerate}
		\item $(X,T)$ is weakly mixing;
		\item for every $n\geq 2$, $S_n(X,T)\cup \Delta_n(X)=X^n$;
		\item there exists $n\geq 2$ such that $S_n(X,T)\cup \Delta_n(X)=X^n$.
	\end{enumerate}
\end{lem}

\begin{prop}\label{prop:weak-inf}
	Let $(X,T)$ be a non-trivial weakly mixing system. Then $(X,T)$ is broken $\calf_{inf}$-$n$-sensitive for any $n\geq 2$.
\end{prop}

\begin{proof}
	Fix $n\geq 2$. Since $(X,T)$ is weakly mixing, by Proposition~\ref{lem:weakly-mixing-n},  $(X^n,T^{(n)})$ is transitive.
	Thus there exists $(x_1,\dotsc, x_n)\in \tran(X^n,T^{(n)})\setminus \Delta^{(n)}(X)$,
	then $(x_1,\dotsc, x_n)$ is a recurrent point in $(X^n,T^{(n)})$.  
	By lemma~\ref{lem:Sn-Wm}, $(x_1,\dotsc, x_n)$ is an essential $n$-sensitive tuple of $(X,T)$.
	By Theorem~\ref{thm:inf} $(X,T)$ is broken $\calf_{inf}$-$n$-sensitive.
\end{proof}

Weak mixing does not imply $\calf_{ps}$-$n$-sensitivity and $\calf_{pubd}$-$n$-sensitivity in general, but we show that
weakly mixing system with a dense set of minimal points (resp. full support) is $\calf_{ps}$-$n$-sensitive (resp. $\calf_{pubd}$-$n$-sensitive).
First we need the following two lemmas.

\begin{lem}\label{lem:M-ps}\cite[Lemma 2.1]{HY05}
	Let $(X,T)$ be a transitive system and $x\in \tran(X,T)$. Then $(X,T)$ is an M-system if and only if for each neighbourhood $U$ of $x$, $N(x,U)\in\calf_{ps}$.
\end{lem}

\begin{lem}\label{lem:E-banach}\cite[Lemma 3.6]{HKY07}
	Let $(X,T)$ be a transitive system and $x\in \tran(X,T)$. Then $(X,T)$ is an E-system if and only if for each neighbourhood $U$ of $x$, $N(x,U)\in\calf_{pubd}$.
\end{lem}

Using the lemmas we can characterize
broken $\calf_{ps}$-$n$-sensitivity and broken $\calf_{pubd}$-$n$-sensitivity in weakly mixing systems.

\begin{prop}\label{prop:weak-M-ps}
	Let $(X,T)$ be a non-trivial weakly mixing $M$-system. Then $(X,T)$ is broken $\calf_{ps}$-$n$-sensitive for any $n\geq 2$.
\end{prop}

\begin{proof}
	Fix $n\geq 2$. Since $(X,T)$ is weakly mixing $M$-system, by Proposition~\ref{lem:weakly-mixing-n} and Lemma~\ref{lem: XY-minimal-points},  $(X^n,T^{(n)})$ is a transitive $M$-system.
	Thus for any $(x_1,\dotsc, x_n)\in \tran(X^n,T^{(n)})\setminus \Delta^{(n)}(X)$,
	$(x_1,\dotsc, x_n)$ is an $\calf_{ps}$-recurrent point by Lemma~\ref{lem:M-ps}. By Lemma~\ref{lem:Sn-Wm}, $(x_1,\dotsc, x_n)$ is an essential $n$-sensitive tuple of $(X,T)$. Therefore by
	Theorem~\ref{thm:broken=} $(X,T)$ is broken $\calf_{ps}$-$n$-sensitive.
\end{proof}

Similar with the proof of Proposition~\ref{prop:weak-M-ps}, we obtain
a characterization for
broken $\calf_{pubd}$-$n$-sensitivity by using Lemma~\ref{lem:E-banach}.

\begin{prop}\label{prop:weak-E-pubd}
	Let $(X,T)$ be a non-trivial weakly mixing $E$-system. Then $(X,T)$ is broken $\calf_{pubd}$-$n$-sensitive for any $n\geq 2$.
\end{prop}

Now we show a necessary and sufficient condition for a non-trivial weakly mixing system to be broken $\calf_{ps}$-$n$-sensitive.

\begin{prop}\label{prop:ps-AP}
	Let $(X,T)$ be a non-trivial weakly mixing system and $n\geq 2$. Then $(X,T)$ is broken $\calf_{ps}$-$n$-sensitive if and only if 
	$(X,T)$ has at least $n$ minimal points.
\end{prop}
\begin{proof}
	$(\Rightarrow)$
	By Theorem~\ref{main1},
	there is an essential $n$-sensitive tuple $(x_1,\dotsc, x_{n})$ which is a minimal point of $(X^{n}, T^{(n)})$. Thus $(X,T)$ has at least $n$ minimal points.
	
	$(\Leftarrow)$ Let $x_1,\dotsc,x_n$ be the pairwise distinct minimal points of $(X,T)$
	and $Z=\overline{\orb(x_1,T)}\times\dotsb\times\overline{\orb(x_n,T)}$, it is clear that $Z$ is $T^{(n)}$-invariant and by Lemma~\ref{lem: XY-minimal-points} $(Z,T^{(n)})$ has a dense set of minimal points. Thus we can choose  $y_i\in X$ for $i=1,\dotsc,n$ such that $(y_1,\dotsc,y_n)$ is a minimal point of $\overline{\orb(x_1,T)}\times\dotsb\times\overline{\orb(x_n,T)}$ and they are pairwise distinct.
	Since $(X,T)$ is weakly mixing,
	by Lemma~\ref{lem:Sn-Wm} $(y_1,\dotsc, y_n)$ is an essential $n$-sensitive tuple of $(X,T)$. Then by 
	Theorem~\ref{main1},
	$(X,T)$ is broken $\calf_{ps}$-$n$-sensitive.
\end{proof}

Now we show a necessary and sufficient condition for a non-trivial weakly mixing system to be broken $\calf_{pubd}$-$n$-sensitive.

\begin{prop}\label{prop:weak-pubd-n}
	Let $(X,T)$ be a non-trivial weakly mixing system and $n\geq 2$. Then the following statements are equivalent:
	\begin{enumerate}
		\item $(X,T)$ is broken $\calf_{pubd}$-$n$-sensitive;
		\item $(X,T)$ has at least $n$ $\calf_{pubd}$-recurrent points;
		\item $|\supp(X,T)|\geq n$.
	\end{enumerate}
\end{prop}
\begin{proof}
	$(1)\Rightarrow(2)$ By Theorem~\ref{thm:broken=} there exists an essential $n$-sensitive tuple which is
	an $\calf_{pubd}$-recurrent point of $(X^n, T^{(n)})$. Thus $(X,T)$ has at least $n$ $\calf_{pubd}$-recurrent points.
	
	$(2)\Rightarrow(3)$
	Let $x_1,\dotsc,x_n$ be the pairwise distinct points which are $\calf_{pubd}$-recurrent points of $(X,T)$.
	By Lemma~\ref{lem:supp-pld}(1),
	$x_1,\dotsc,x_n\in\supp(X,T)$. Thus 
	$|\supp(X,T)|\geq n$.
	
	$(3)\Rightarrow(1)$
	Let $x_1,\dotsc,x_n$ be the pairwise distinct points of $\supp(X,T)$,
	then $(x_1,\dotsc,x_n)\in \supp(X,T)\times\dotsc\times\supp(X,T)=\supp(X^n,T^{(n)})$ by Lemma~\ref{lem:supp-pld}(2). 
	Let 
	$$3\delta=\min_{1\leq i<j\leq n}d(x_i, x_j).$$
	By Lemma~\ref{lem:supp-pld}(1),
	there exists an $\calf_{pubd}$-recurrent point $(y_1,\dotsc,y_n)$ of $(X^n, T^{(n)})$ such that 
	$$(y_1,\dotsc,y_n)\in  B(x_1,\delta) \times\dotsb\times B(x_n,\delta)$$ then  $y_1,\dotsc,y_n$ are pairwise distinct.
	Since $(X,T)$ is weak mixing,
	by Lemma~\ref{lem:Sn-Wm} $(y_1,\dotsc, y_n)$ is an essential $n$-sensitive tuple of $(X,T)$.
	By Theorem~\ref{thm:broken=},
	$(X,T)$ is broken $\calf_{pubd}$-$n$-sensitive.
\end{proof}

For minimal systems, we proved that 
$(X,T)$ is broken $\calf_{pubd}$-sensitive if and only if it is not mean equicontinuous.
Now we prove a similar result for weakly mixing systems.

\begin{prop}\label{prop:weak-pubd-supp}
	Let $(X,T)$ be a weakly mixing system. Then the following statements are equivalent:
	\begin{enumerate}
		\item $(X,T)$ is mean equicontinuous;
		\item $\supp(X,T)$ is a sigleton;
		\item $(X,T)$ is not broken $\calf_{pubd}$-sensitive.
	\end{enumerate}
\end{prop}

\begin{proof}
	$(2)\Leftrightarrow(3)$
	By Proposition~\ref{prop:weak-pubd-n}.
	
	$(1)\Rightarrow(2)$
	Since $(X,T)$ is weakly mixing, 
	by Lemma~\ref{lem:Sn-Qn} and  Lemma~\ref{lem:Sn-Wm},
	$Q(X,T)\supset S_2(X,T)\cup\Delta_2=X\times X$
	 and by Lemma~\ref{lem: Q-P-BP},
	$Q(X,T)=P(X,T)=BP(X,T)=X\times X$, 
	by \cite[Theorem $5.2$]{LT14} 
	$\supp(X,T)$ is a singleton.
	
	$(2)\Rightarrow(1)$
	By \cite[Theorem $5.2$]{LT14},  
	$BP(X,T)=X\times X$, thus $Q(X,T)=P(X,T)=BP(X,T)=X\times X$. By 
	Lemma~\ref{lem: Q-P-BP}, $(X,T)$ is mean equicontinuous.
\end{proof}

\subsection{Examples for broken family sensitivity in weakly mixing systems}
The following implications follow from the definitions, we show that most of them are strict in weakly mixing systems.
\begin{equation*}
\begin{array}{cccc}
\ \ \text{broken $\calf_{ps}$-$n$-sen.} \   \ \   \Rightarrow &  \ \ \ \text{broken $\calf_{pubd}$-$n$-sen.} \   \ \  \Rightarrow &  \text{broken $\calf_{inf}$-$n$-sen.}\\
\Uparrow & \Uparrow & \Updownarrow  \\
\text{broken $\calf_{ps}$-$(n+1)$-sen.} \Rightarrow & \text{broken $\calf_{pubd}$-$(n+1)$-sen.} \Rightarrow &  \text{broken $\calf_{inf}$-$(n+1)$-sen.}
\end{array}
\end{equation*}

\begin{exam}[broken $\calf_{pubd}$-$n$-sensitivity $\not\Rightarrow$ broken $\calf_{ps}$-sensitivity]
	In \cite[Theorem 1.1]{Lian2017}, there is a weakly mixing system with full support and a unique minimal point.
	From Proposition~\ref{prop:weak-E-pubd} and
	Proposition~\ref{prop:ps-AP} we know that this kind of system is broken $\calf_{pubd}$-$n$-sensitive for any $n\geq 2$ but not broken $\calf_{ps}$-sensitive.
\end{exam}

\begin{exam}[broken $\calf_{inf}$-$n$-sensitivity $\not\Rightarrow$ broken $\calf_{pubd}$-sensitivity]
	It is shown in~\cite{HZ02} that there exists a strongly mixing subshift with $\supp(X,T)$ a singleton.
	By Proposition~\ref{prop:weak-inf} and Proposition~\ref{prop:weak-pubd-supp}, the system is broken $\calf_{inf}$-$n$-sensitive for any $n\geq 2$ and not broken $\calf_{pubd}$-sensitive.
\end{exam}

\begin{exam}[broken $\calf_{ps}$-$n$-sensitivity $\not\Rightarrow$ broken $\calf_{ps}$-$(n+1)$-sensitivity]
	For every $n\geq 2$, by \cite[Proposition 6.3]{HKKPZ18} there exists a non-trivial strongly mixing system with exact $n$ minimal points.
	Then by Theorem~\ref{prop:ps-AP}, this strongly mixing system is broken $\calf_{ps}$-$n$-sensitive, but not
	broken $\calf_{ps}$-$(n+1)$-sensitive.
\end{exam}

\begin{exam}[broken $\calf_{pubd}$-$2$-sensitivity $\not\Rightarrow$ broken $\calf_{pubd}$-$3$-sensitivity]
\medskip
	Inspired by \cite[Proposition 6.3]{HKKPZ18}, 
	We construct a
	weakly mixing system $(X, T)$ with two fixed points, which are the only $\calf_{pubd}$-recurrent points of the system.
	Let $\Sigma=\{0,1\}^{\bbz_+}$ and
	$\sigma :\Sigma \to \Sigma$ be the full (one-sided) shift.
	We are going to find the system $(X, T)$ of the form $(\overline{\orb_{\sigma} (x)}, \sigma)$  for some $x \in \Sigma$.
	
	Now we are going to define $x\in \Sigma$. Let $A_1=10$ and define inductively
	the blocks $A_2, A_3, \dots$,  define $x$ to be the limit of the blocks $A_k$.
	Suppose that we have defined $A_k, k\in \mathbb{N}$. Since $A_{k}$ has finitely many subblocks, there is a finite number of different
	pairs of these subblocks.
	For any pair $(W_1, W_2)$ of subblocks of $A_k$ we will define a block $c (W_1, W_2)$ by using their combination.
	Then we are ready to define $A_{k+1}$: at the beginning of $A_{k+1}$ we write $A_k0^{|A_k|\cdot 10^{(k+1)}}$, and then all possible combination blocks $c (W_1, W_2)$ of
	pairs $(W_1, W_2)$ of subblocks of $A_{k}$ in any fixed order, and then $1^{(k+1)}0^{|A_k|\cdot 10^{(k+1)}}$.
	\emph{The combination block} of the pair $W_1, W_2$, i.e. $c(W_1, W_2)$, is defined as follows:
	\[\begin{split}
	c(W_1, W_2)&=W_1 0^{|A_k|\cdot 10^{(k+1)}} W_2 0^{|A_k|\cdot 10^{(k+1)}} W_1 0^{|A_k|\cdot (10^{(k+1)}+1)} W_2 0^{|A_k|\cdot (10^{(k+1)}+1)}\dotsb
	\\& W_1 0^{|A_k|\cdot 2\cdot 10^{(k+1)}} W_2 0^{|A_k|\cdot 2\cdot 10^{(k+1)}}.
    \end{split}\]

    Put $(X, T):=(\overline{\orb_{\sigma} (x)},\sigma)$. 
    Recall that the base for the open sets in $\Sigma$ is given by the collection of all cylinder sets
    $[c_0c_1c_2 \dots c_m]=\{x \in \Sigma:~ x_i=c_i \text{ for } 0\le i\leq m \}.$
    We are going to prove that $N(U,V)$ is a thick set for any opene subsets $U,V\subset X$. 
    Fix opene subsets $U$, $V$ of $X$.
    Clearly there exist subblocks $W_1$ and $W_2$ of $x$ such that $U\supset [W_1]\cap X$ and $V\supset [W_2]\cap X$, and there exists $m\in\bbn$ such that both  $W_1$ and $W_2$ are subblocks of $A_m$. For any $k\geq m$, both  $W_1$ and $W_2$ are subblocks of $A_k$,
    there exists a combination block $c(W_1, W_2)$
    containing the subblock $W_1 0^{|A_k|\cdot 10^{(k+1)}} W_2$, 
    $W_1 0^{|A_k|\cdot (10^{(k+1)}+1)} W_2 0^{|A_k|\cdot (10^{(k+1)}+1)}$, $\dotsc$,
    $W_1 0^{|A_k|\cdot 2\cdot 10^{(k+1)}} W_2 0^{|A_k|\cdot 2\cdot 10^{(k+1)}}$,
    and hence $N(U, V)\supset N([W_1]\cap X, [W_2]\cap X)\supset \{|A_k|\cdot 10^{(k+1)}+ |W_1|, |A_k|\cdot 10^{(k+1)}+ |W_1|+ 1, \dots, |A_k|\cdot 2\cdot 10^{(k+1)}+ |W_1|\}$, which implies that $N(U,V)$ is a thick set and 
    thus $(X, T)$ is weakly mixing.
    It is easy to see that for each $k\in \mathbb{N}$, $0^{|A_k|\cdot 10^{(k+1)}}$ and $1^{(k+1)}$ appear in $x$.
	In particular, both
	$0^\infty$ and $1^\infty$ are fixed points of $(X,T)$.
	Collapsing $0^\infty$ and $1^\infty$ to a point, we get a factor map $\pi:(X, T)\to (Y, S)$ such that $\pi(0^{\infty})=\pi(1^{\infty})$ and $\pi^{-1}(\pi(z))=z$ for any $z\in X\setminus\{0^{\infty}, 1^{\infty}\}$.
	Then $\pi(0^{\infty})$ is a fixed point of $(Y, S)$. 
	Now let us check that $\pi(0^{\infty})$ is the unique $\calf_{pubd}$-recurrent point of $(Y,S)$.
    Actually, in any subblock of $x$ with length more than $|A_k|\cdot 10^{(k+1)}$ one can find that the distribution probability of the blocks of consecutive 1's or 0's of length $|A_k|$ is greater than or equal to $1-\frac{3}{10^{(k+1)}}$. That is, for
	any fixed $k\in\bbn$, denote $n=|A_k|$, 
	for any subblock $w=w_1w_2\dots w_l$ of $x$ with $l\ge |A_k|\cdot 10^{(k+1)}$, we have
	\[\frac{|\{i:w_{i+1}w_{i+2}\dotsc w_{i+n}=1^n\ \text{or}\ 0^n\}|}{l-n+1}\geq 1-\frac{3|A_k|}{l-n+1}
	\geq1-\frac{3}{10^{(k+1)}-1}.\]
	Then for any $y\in Y$ and $m\in\bbn$, 
	\[\frac{|\{i: d_Y(S^i(y), S^i(\pi(0^{\infty}))<\frac{1}{2^n}
	\}\cap[m,m+l-n]|}{l-n+1}
	\geq 1-\frac{3}{10^{(k+1)}-1}.\]
	That is, for any $y\in Y$ and neighbourhood $W$ of $\pi(0^\infty)$, $N(y,W)\in \calf_{lbd1}$.
	By Lemma~\ref{lem:only-recurrent}, 
	$\pi(0^\infty)$ is the unique $\calf_{pubd}$-recurrent point of $(Y,S)$, and then $0^\infty$ and $1^\infty$ are the only $\calf_{pubd}$-recurrent points of $(X,T)$. By Proposition~\ref{prop:weak-pubd-n}
    $(X,T)$ is broken $\calf_{pubd}$-$2$-sensitive but not broken $\calf_{pubd}$-$3$-sensitive.
\end{exam}


\begin{thebibliography}{99}

	\bibitem{Ethan-01} E. Akin and E. Glasner,
	\textit{Residual properties and almost equicontinuity},
	J. Anal. Math. \textbf{84}  (2001), 243--286.
	
	\bibitem{Auslander1980}
	J. Auslander and J. Yorke, \textit{Interval maps, factors of maps, and
		chaos}, T\^ohoku Math. J. (2) \textbf{32} (1980), no.~2, 177--188.
	
	\bibitem{Blokh02}
	A. Blokh and A. Fieldsteel, \textit{Sets that force recurrence},
	Proc. Amer. Math. Soc. \textbf{130} (2002), no. 12, 3571--3578.
	
	\bibitem{Downarowicz2005} T. Downarowicz,
	\textit{Survey of odometers and Toeplitz flows},
	Algebraic and topological dynamics, 7–37,
	Contemp. Math., 385, Amer. Math. Soc., Providence, RI, 2005.
	
	\bibitem{Downarowicz2016} T. Downarowicz and E. Glasner,
	\textit{Isomorphic extensions and applications},
	Topol. Methods Nonlinear Anal. \textbf{48} (2016), no. 1, 321–338.
	
	\bibitem{Fomin51} S. Fomin, \textit{On dynamical systems with a purely point spectrum}, Doklady Akad. Nauk SSSR (N.S.) \textbf{77}, (1951). 29--32.
	
	\bibitem{F67}  H. Furstenberg, \textit{Disjointness in ergodic theory, minimal sets, and a problem in Diophantine approximation}, Math. Syst. Theory \textbf{1} (1967), 1--49.

	\bibitem{Glasner1989} E. Glasner, \emph{Book Review: Minimal flows and their extensions},  Bull. Amer. Math. Soc. (N.S.) \textbf{21} (1989), no. 2, 316--319.
	
	\bibitem{HZ02} W. He and Z. Zhou, \textit{A topologically mixing system whose measure center is a singleton} (Chinese),  Acta Math. Sinica (Chin. Ser.) \textbf{45} (2002), no. 5, 929--934.
	
	\bibitem{HY05} W. Huang and X. Ye,
	\textit{Dynamical systems disjoint from any minimal system}, Trans. Amer. Math. Soc. \textbf{357} (2005), no. 2, 669--694.
	
	\bibitem{HKY07} W. Huang, K. Park and X. Ye,
	\textit{Topological disjointness from entropy zero systems},
	Bull. Soc. Math. France \textbf{135} (2007), no. 2, 259--282.
	
	\bibitem{HLY11} W. Huang, P. Lu and X. Ye,
	\textit{Measure-theoretical sensitivity and equicontinuity},
	Israel J. Math. \textbf{183} (2011), 233--283.
	
	\bibitem{HLY12} W. Huang, H. Li and X. Ye,
	\textit{Family independence for topological and measurable dynamics},
	Trans. Amer. Math. Soc. \textbf{364} (2012), no. 10, 5209--5242.
	
	\bibitem{HKKPZ16}
	W. Huang, D. Khilko, S. Kolyada and G. Zhang,
	\textit{Dynamical compactness and sensitivity}, J. Differential Equations \textbf{260} (2016), no. 9, 6800--6827.
	
	\bibitem{HKKPZ18}
	W. Huang, D. Khilko, S. Kolyada, A. Peris and G. Zhang,
	\textit{Finite intersection property and dynamical compactness}, J. Dynam. Differential Equations \textbf{30} (2018), 1221--1245.
	
	\bibitem{HKZ18}
	W. Huang, S. Kolyada and G. Zhang, \textit{Analogues of Auslander-Yorke theorems for multi-sensitivity}, Ergodic Theory Dynam. Systems \textbf{38} (2018), 651--665.
	
	\bibitem{L12} J. Li, \textit{Dynamical characterization of C-sets and its application}, Fund. Math.
	\textbf{216} (2012), 259--286.
	
	\bibitem{LT14} J. Li and S. Tu, \textit{On proximality with Banach density one}, J. Math. Anal.Appl. \textbf{416} (2014), 36--51.
	
	\bibitem{LTY15} J. Li, S. Tu and X. Ye, \textit{Mean equicontinuity and mean sensitivity},
	Ergodic Theory Dynam. Systems \textbf{35} (2015), 2587--2612.

	\bibitem{Li-Yang1} J. Li and Y. Yang,
	\textit{Stronger versions of sensitivity for minimal group actions}, Acta Math. Sin. (Engl. Ser.), \textbf{37} (2021), 1933--1946.
	
	\bibitem{Li-Yang2} J. Li and Y. Yang,
	\textit{On $n$-tuplewise IP-sensitivity and thick sensitivity}, Discrete Contin. Dyn. Syst., to appear, 
	DOI:10.3934/dcds.2021211.
	
	\bibitem{Li-Ye} J. Li and X. Ye, \textit{Recent development of chaos theory in topological dynamics},
	 Acta Math. Sin. (Engl. Ser.)  \textbf{32} (2016), no. 1, 83--114.
	
	\bibitem{LYY21} J. Li, X. Ye and T. Yu, \textit{Mean equicontinuity, complexity and applications}, Discrete Contin. Dyn. Syst. \textbf{41} (2021), 359--393.
	
	\bibitem{LShi14} R. Li and Y. Shi,
	\textit{Stronger forms of sensitivity for measure-preserving maps and semiflows on probability spaces}, Abstr. Appl. Anal. 2014, Art. ID 769523, 10 pp.
	
	\bibitem{Lian2017} Z. Lian, S. Shao and  X. Ye,
	\emph{Weakly mixing proximal topological models for ergodic systems and applications}, Fund. Math. \textbf{236} (2017), no. 2, 161--185.
	
	\bibitem{MShao07} A. Maass, S. Shao, \textit{Structure of bounded topological-sequence-entropy minimal systems},
	J. Lond. Math. Soc. (2) \textbf{76} (2007), no. 3, 702--718.
	
	\bibitem{M07} T. K. Subrahmonian Moothathu, \textit{Stronger forms of sensitivity for dynamical systems},
	Nonlinearity \textbf{20} (2007), 2115--2126.
	
	\bibitem{QZ18}
	J. Qiu and J. Zhao, \textit{A note on mean equicontinuity}, J. Dynam. Differential Equations \textbf{32} (2020), 101--116.
	
	\bibitem{Ruelle1977} D. Ruelle, \textit{Dynamical systems with turbulent behavior}, Mathematical problems in theoretical
	physics (Proc. Internat. Conf., Univ. Rome, Rome, 1977), Lecture Notes in Phys., vol. 80,
	Springer, Berlin-New York, 1978, pp. 341--360.
	
	\bibitem{SYZ} S. Shao, X. Ye and R. Zhang, \textit{Sensitivity and regionally proximal relation in minimal systems},
	 Sci. China Ser. A \textbf{51} (2008), 987--994.
	
	\bibitem{TZ11} F. Tan and R. Zhang, \textit{On $\mathcal{F}$-sensitive pairs}, Acta Math. Sci. Ser. B (Engl. Ed.) \textbf{31} (2011), 1425--1435.
	
	\bibitem{X05} J. Xiong, \textit{Chaos in topological transitive systems},
	 Sci. China Ser. A \textbf{48} (2005), 929--939.
	
	\bibitem{Y18} X. Ye and T. Yu,
	\textit{Sensitivity, proximal extension and higher order almost automorphy},
	Trans. Amer. Math. Soc.  \textbf{370}  (2018), no. 5, 3639–3662.
	
	\bibitem{YZ08} X. Ye and R. Zhang,
	\textit{On sensitive sets in topological dynamics},
	Nonlinearity \textbf{21}  (2008),  no. 7, 1601–1620.
	
	\bibitem{Zou17} Y. Zou,
	\textit{Stronger version sensitivity, almost finite to one extension and maximal pattern	entropy},
	Commun. Math. Stat. \textbf{5} (2017), no. 2, 123--139.
	
\end{thebibliography}
\end{document}